\documentclass[a4paper,10pt,english]{article}
\usepackage{babel}
\usepackage[iso-8859-7]{inputenc}
\usepackage{graphicx}
\usepackage{amsmath}
\usepackage{amsfonts}
\usepackage{amssymb}%
\newtheorem{theorem}{Theorem}

\newtheorem{corollary}{Corollary}

\newtheorem{lemma}{Lemma}

\newtheorem{remark}[theorem]{Remark}

\newenvironment{proof}[1][Proof]{\textbf{#1.} }{\  \rule{0.5em}{0.5em}}

\begin{document}

\title{Projections from surfaces of revolution in the Euclidean plane}
\author{C. Charitos and P. Dospra}
\maketitle

\begin{abstract}
For a specific class of surfaces of revolution $S$, the existence of a smooth
map $\Phi$ from a neighbourhood $U$ of $S$ to the Euclidean plane $E^{2}$
preserving distances infinitesimally along the meridians and the parallels of
$S$ and sending the meridional arcs of $U\cap S$ to straight lines of $E^{2},$
is proven.\smallskip


\end{abstract}

\section{Introduction}

In \cite{[Euler]} (see \cite{[Heine]} for the translation of \cite{[Euler]} in
English) Euler proved that there does not exist a perfect map from the sphere
$S^{2}$ or from a part of $S^{2}$ to the Euclidean plane $E^{2}.$ Recall that
a smooth map $f$ from $S^{2}$ (or from a part of $S^{2})$ to $E^{2}$ is called
\textit{perfect} if for each $p\in S^{2}$\ there is a neighborhood $U(p)$\ of
$p$ in $S^{2}$ such that the restriction of $f$\ on $U(p)$\ preserves
distances infinitesimally along the meridians and the parallels of $S^{2}$ and
$f$ preserves also angles between meridians and parallels \cite{[CP]}. In
modern geometric language a perfect map is a local isometry from $S^{2}$ to
$E^{2}$ and thus Euler's theorem results as a corollary of the Gauss Egregium
Theorem which was proven many years later. However, Euler's method of proof is
very fruitful and can be applied in similar problems, see for instance
Proposition 5 in \cite{[CP]}. Very briefly, Euler's basic idea for the
non-existence of a perfect map from $S^{2}$ to $E^{2},$ is to translate
geometrical conditions to a system of differential equations and prove that
this system does not have a solution. Using Euler's method, the non-existence
of a smooth map from a neighbourhood $U$ of $S^{2}$ to $E^{2}$ which preserves
distances infinitesimally along the meridians and the parallels of $S^{2}$ and
which sends the meridional arcs of $U\cap S^{2}$ to straight lines of $E^{2},$
can also be proven \cite{[CP]}.

The origin of all these problems lies in the ancient problem of cartography,
that is, the problem of constructing geographical maps from $S^{2}$ (or from a
subset of $S^{2})$ to $E^{2}$ which satisfy certain specific requirements.
This problem can also be considered as part of a more general subject
which explores the existence of coordinate transformations that preserves some
geometrical properties from one coordinate system to another.
Several prominent mathematicians have been studying this problem from
antiquity to our days and in the course of this study, $S^{2}$ was replaced
gradually by surfaces of revolution or by surfaces in $E^{3}$ generally, see
\cite{[Papadop]} for an excellent historical recursion on this subject.

The goal of this work is to determine the class of surfaces of revolution $S$
for which there exists a smooth map $\Phi$ from an open neighbourhood $U$ of
$S$ to $E^{2}$ preserving distances infinitesimally along the meridians and
the parallels of $S$ and sending meridional arcs of $U\cap S$ to straight
lines of $E^{2}.$ Furthermore the map
$\Phi$ is computed explicitly. For the computation of $\Phi$ we follow Euler's
ideas, that is, we convert geometrical conditions to differential equations
whose solutions allow us to find $\Phi.$

As a corollary of the above result we deduce that if $p$ is a point of $S$ and
if the Gaussian curvature at $p$ is positive, then a map $\Phi$ as above does
not exist in a neighbourhood $U$ of $p.$ We also deduce that if $S_{0}$ is an
abstract surface of constant negative curvature, such maps $\Phi$ do not exist
from an open subset $U$ of $S_{0}$ to $E^{2}.$

\section{Statement of Results}

Let $S$ be  a surface of revolution  in  $E^{3}$.  
In the following we assume that all maps are of class $C^s$, $s\geq2$.   
 We consider a
parametrization of $S$ given by
\begin{equation}
r(t,u)=(f(u)\cos t,f(u)\sin t,g(u)), \tag{r}\label{r}%
\end{equation}
where $f(u)>0,$ $a<u<b,$ $t\in \lbrack0,2\pi].$ In what follows we will always
assume that the curve $\gamma(u)=(f(u),g(u))$ is parametrized by arc-length
i.e. $(f^{\prime})^{2}(u)+(g^{\prime})^{2}(u)=1.$ Therefore the Riemannian
metric on $S$ takes the form
\[
ds^{2}=du^{2}+f^{2}(u)dt^{2}.
\]
For $t$ fixed, the $u$-curves $r(t,\cdot)$ are called \textit{meridians} of
$S$ and are geodesics for the metric $ds,$ while for $u$ fixed, the $t$-curves
$r(\cdot,u)$ are called \textit{parallels} and are not geodesics in general.

At each point $p\in S$ we may assign a unique pair of coordinates
$(t,u)\in(a,b)\times \lbrack0,2\pi)$ since $p=r(t,u).$ Thus, $p$ will be
identified with this pair $(t,u).$

\begin{theorem}
\label{main theorem} Assume that $f^{\prime}(u)\neq0$ and
$f^{\prime \prime}(u)\neq 0$ for each $u\in(a,b)$. Let
$U$ be an open connected
subset of $S$. Then, there
exists a map $\Phi:U\rightarrow
E^{2}$ having the properties \\
$(C1):$ $\Phi$ sends the meridional arcs of $S\cap U$ to straight lines
of $E^{2};$\\
$(C2):$ $\Phi$ preserves distances infinitesimally along the meridians
and the parallels of $S;$\\
if and only if $f^{2}=cu^{2}+du+k$, where $c,d,k$ are constants with $k>0$ and $c>0$.
Furthermore, assuming that $\Phi(t,u)=(x(t,u),y(t,u))$, we have that
\begin{align*}
x(t,u)  &  =u\cos(b(t))+\int \sqrt{k}\cos(\theta_{0}-b(t))dt,\\
y(t,u)  &  =-u\sin(b(t))+\int \sqrt{k}\sin(\theta_{0}-b(t))dt,
\end{align*}%
where
\[
b(t)=-\sqrt{c}t+c_{0}.
\]
\end{theorem}

\begin{figure}[h] 
\centering{\includegraphics[scale = 0.3]{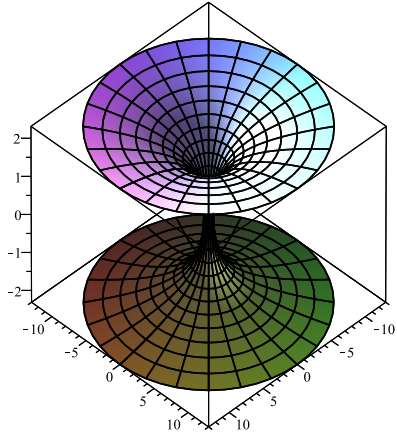}
\caption{\small Surface with $f^{2}=u^{2}+1$}}
\end{figure}

In Figure 1, a surface of revolution $S$ is drawn which satisfies all
hypothesis of Theorem \ref{main theorem}. In this example, taking $f^{2}
=u^{2}+1$ it results that $g(u)=\ln(\sqrt{u^{2}+1}+u),$ $u>0.$ The picture
confirms that the Gaussian curvature of each point of $S$ is negative.

\begin{corollary}
(1) If $p\in S$ and the Gaussian curvature at $p$ is positive then for each
neighborhood $U$ of $p$ in $S$ there does not exist a map $\Phi:U\rightarrow
E^{2}$ satisfying the conditions $(C1)$ and $(C2).$

(2) If $S_{0}$ is a Riemannian surface of constant negative curvature then for
each open neighbourhood $U\subset S_{0}$ there does not exist a map
$\Phi:U\rightarrow E^{2}$ satisfying the conditions $(C1)$ and $(C2).$
\end{corollary}

Condition $(C1)$ is a natural requirement since meridians are geodesics of $S$
and thus it is required to be sent to geodesics of $E^{2}$ via $\Phi.$

Condition $(C2)$ appears in Euler's writings and means that the elementary
length between two points $p,$ $q$ on a meridian (resp. two points $p,r$ on a
parallel) of $S$ is equal to the elementary length of points $P=\Phi(p),$
$Q=\Phi(q)$ (resp. of points $P,$ $R=\Phi(r)).$ In other words, the `elements'
$pq$ and $pr$ are equal to the `elements' $PQ$ and $PR$ respectively
(following the terminology of \cite{[Euler]}). In order to express $(C2)$
rigorously, let $p=(t,u),$ $q=(t,u+du),$ $r=(t+dt,u),$ $P=\Phi(t,u),$
$Q=\Phi(t,u+du),$ $R=\Phi(t+du,u).$ If we denote by $|p_{1}-p_{2}|$ the
distance between the points $p_{1},$ $p_{2}\in S$ and by $||P_{1}-P_{2}||$ the
Euclidean distance between the points $P_{1},$ $P_{2}\in E^{2},$ then
condition $(C2)$ means that:%
\begin{equation}
\underset{du\rightarrow0}{\lim}\frac{|q-p|}{|du|}=\underset{du\rightarrow
0}{\lim}\frac{||P-Q||}{|du|} \tag{i}\label{i}%
\end{equation}

\begin{equation}
\underset{dt\rightarrow0}{\lim}\frac{|r-p|}{|du|}=\underset{dt\rightarrow
0}{\lim}\frac{||R-Q||}{|du|}. \tag{ii}\label{ii}%
\end{equation}

Using the coordinate functions $x(t,u)$ and $y(t,u)$ the equalities (\ref{i})
and (\ref{ii}) take respectively the following form:%

\begin{equation}
\sqrt{(\frac{\partial x}{\partial u})^{2}+(\frac{\partial y}{\partial u})^{2}%
}=1, \tag{iii}\label{iii}%
\end{equation}

\begin{equation}
\text{\qquad}\sqrt{(\frac{\partial x}{\partial t})^{2}+(\frac{\partial
y}{\partial t})^{2}}=f(u). \tag{iv}\label{iv}%
\end{equation}
Indeed, relations (\ref{iii}) and (\ref{iv}) correspond to the relations (I)
and (II) of (\cite{[Heine]}, p. 5). In \cite{[CP]}, these relations are
reproved using a modern mathematical language and they are labelled as
relations (6) and (7) respectively. In the present work, relations (\ref{iii})
and (\ref{iv}) are obtained by replacing $\cos u$ in the parametrization
$(\cos u\cos t,\cos u\sin t,\sin u)$ of $S^{2}$ by the function $f(u)$ in the
parametrization $(f(u)\cos t,f(u)\sin t,g(u))$ of $S$ and then repeating the
same steps. As a result, using Euler's method, condition $(C2)$ is translated
in a system of differential equations consisting of the relations (\ref{iii})
and (\ref{iv}).

Combining $(C1)$ and $(C2)$ we deduce that $\Phi$ restricted to a meridian of
$S$ is an isometry onto its image.

\begin{remark} {\rm
If $f^{\prime}(u)=0$ for each $u\in(a,b)$, then, the curve $\gamma
(u)=(f(u),g(u))$ is a straight line in the $(x,z)$-plane and so, the surface
$S$ obtained by revolving $\gamma$ about the $z$-axis is Euclidean i.e.
locally isometric to the Euclidean plane $E^{2}.$
Furthermore, if $f^{\prime
\prime}(u)=0$ for each $u,$ we deduce that $f^{\prime}$ and $g^{\prime}$ are
constant functions since by assumption the curve $\gamma(u)=(f(u),g(u))$ is
parametrized by arc-length. Therefore $\gamma(u)$ is a line segment and thus
$S$ is a locally isometric to $E^{2}.$ }
\end{remark}

\section{Auxiliary Lemmas}

In this section we give some results that we will use for the proof of Theorem 1.

\begin{lemma}
 Assume that $f^{\prime}(u)\neq 0$ and
$f^{\prime \prime}(u)\neq 0$, for each $u\in(a,b)$. Let
$U$ be an open connected
subset of $S$ and  a map $\Phi:U\rightarrow
E^2$ having the properties $(C_1)$ and $(C_2)$. Then, 
 there are
variables $\phi$ and $\omega$ which are functions of $t,$ $u$ satisfying
\begin{equation*}
\left(\frac{\partial x}{\partial u},\frac{\partial y}{\partial u}\right)=
(\cos \phi
,\sin \phi), 
\ \ \ \
\left(\frac{\partial x}{\partial t},\frac{\partial y}{\partial t}\right)=(f(u)\cos
\omega,f(u)\sin \omega)
\end{equation*}
and
\[
f^{\prime \prime}(u)=\left(  \frac{\partial \omega}{\partial u}(t,u)\right)
^{2}f(u).  
\]
\end{lemma}
\begin{proof}
Proceeding as in the proof of Proposition 5 of \cite{[CP]} 
we have that there are
variables $\phi$ and $\omega$ which are functions of $t,$ $u$ such that
\begin{equation*}
(\frac{\partial x}{\partial u},\frac{\partial y}{\partial u})=(\cos \phi
,\sin \phi),  
\end{equation*}
\begin{equation*}
(\frac{\partial x}{\partial t},\frac{\partial y}{\partial t})=(f(u)\cos
\omega,f(u)\sin \omega).  
\end{equation*}

Since
\[
\frac{\partial^{2}}{\partial u\partial t}=\frac{\partial^{2}}{\partial
t\partial u},
\]
we have%
\begin{equation*}
-\sin \phi \cdot \frac{\partial \phi}{\partial t}=f^{\prime}\cdot \cos \omega
-f\cdot \sin \omega \cdot \frac{\partial \omega}{\partial u}. 
\end{equation*}

\begin{equation*}
\cos \phi \cdot \frac{\partial \phi}{\partial t}=f^{\prime}\cdot \sin \omega
+f\cdot \cos \omega \cdot \frac{\partial \omega}{\partial u}. 
\end{equation*}
Multiplying the first of the above equality by 
 $\cos \omega$, the second by $\sin \omega$ and
adding, we deduce that
\begin{eqnarray*}
\lefteqn{
(-\sin \phi \cdot \cos \omega+\cos \phi \cdot \sin \omega)\frac{\partial \phi}{\partial
t}=}\\
& & f^{\prime}\cdot \cos^{2}\omega-f\cdot \sin \omega \cdot \cos \omega \cdot
\frac{\partial \omega}{\partial u}+f^{\prime}\cdot \sin^{2}\omega+f\cdot
\cos \omega \cdot \sin \omega \cdot \frac{\partial \omega}{\partial u}%
\end{eqnarray*}
if and only if 
\begin{equation*}
\sin(\phi-\omega)\cdot \frac{\partial \phi}{\partial t}=f^{\prime}.
\end{equation*}
Similarly, multiplying the first equality by
 $\cos \phi$, the second by $\sin \phi$
and adding, we obtain%
\begin{eqnarray*}
\lefteqn{\sin u\cdot \cos \omega \cdot \cos \phi-}\\
& & \cos u\cdot \sin \omega \cdot \frac{\partial \omega}{\partial u}\cdot \cos \phi-\sin u\cdot \sin \omega \cdot \sin\phi+
\cos u\cdot \cos \omega \cdot \frac{\partial \omega}{\partial u}\cdot \sin \phi = 0
\end{eqnarray*}
which implies that%
\begin{equation*}
f\cdot \sin(\phi-\omega)\frac{\partial \omega}{\partial u}=-f^{\prime}\cdot
\cos(\phi-\omega). 
\end{equation*}

On the other hand, the condition $(C1)$ implies that the meridians are
 mapped  to straight lines, and so, we have 
\begin{equation*}
\frac{\partial \phi}{\partial u}=0.
\end{equation*}
Thus, 
differentiating 
\begin{equation*}
\sin(\phi-\omega)\cdot \frac{\partial \phi}{\partial t}=f^{\prime}
\end{equation*}
 with respect to $u$ and using the previous equality, we obtain%
\begin{equation*}
-\cos(\phi-\omega)\cdot \frac{\partial \omega}{\partial u}\cdot \frac
{\partial \phi}{\partial t}=f^{\prime \prime}.
\end{equation*}

Multiplying 
\begin{equation*}
\sin(\phi-\omega)\cdot \frac{\partial \phi}{\partial t}=f^{\prime}
\end{equation*}
by%
\[
\frac{\partial \omega}{\partial u}\cdot \frac{\partial \phi}{\partial t}%
\]
we have%
\[
f\cdot \sin(\phi-\omega)\cdot\left(\frac{\partial \omega}{\partial u}\right)^{2}\cdot
\frac{\partial \phi}{\partial t}=-f^{\prime}\cdot \cos(\phi-\omega)\cdot
\frac{\partial \omega}{\partial u}\cdot \frac{\partial \phi}{\partial t}.
\]
Hence, combining the above equalities,  we deduce%
\[
f\cdot f^{\prime}\cdot(\frac{\partial \omega}{\partial u})^{2}=f^{\prime}\cdot
f^{\prime \prime}%
\]
which implies that%
\begin{equation*}
f^{\prime}(u)(f^{\prime \prime}(u)-\left(  \frac{\partial \omega}{\partial
u}\right)  ^{2}f(u))=0.
\end{equation*}
Since we have $f^{\prime}(u)\neq 0$, we obtain the result.
\end{proof}

\begin{lemma}
Assume that $f^{\prime}(u)\neq 0$ and
$f^{\prime \prime}(u)\neq 0$, for each $u\in(a,b)$. Assume that
\begin{equation*}
2f^{\prime}a^{\prime}+fa^{\prime \prime}=0.
\end{equation*}
Then, we have $f^{2}=cu^{2}+du+k$ and
\[
a=a(u)=\arctan \left(  \frac{2c}{\sqrt{-\Delta}}\left(  u+\frac{d}{2c}\right)
\right),
\]
where $\Delta=d^{2}-4ck$.
\end{lemma}
\begin{proof} Putting $y=a^{\prime}$, we have the differential equation
\[
y^{\prime}+\frac{2f^{\prime}}{f}y = 0.
\]
Its  solution  is
\begin{equation*}
y=a^{\prime}=Ce^{-2\int \frac{f^{\prime}}{f}du}.
\end{equation*}
It follows that
$$(a^{\prime})^{2}   = C^2  e^{-4\int \frac{f^{\prime}}{f}du},$$
whence
 $$\frac{f^{\prime \prime}}{f}=C^2 e^{-4\int \frac{f^{\prime}}{f}du},$$
and so, we obtain
$$\ln f^{\prime \prime}-\ln f    =K-4\int \frac{f^{\prime}}{f}du.$$
Differentiating the above equality, we get
$$\frac{f^{(3)}}{f^{\prime \prime}}-\frac{f^{\prime}}{f}   =-4\frac{f^{\prime}%
}{f},$$
and therefore we deduce
\[
f^{(3)}f+3f^{\prime \prime}f^{\prime}=0.
\]

On the other hand, we have
\[
(ff^{\prime})^{\prime \prime}=((ff^{\prime})^{\prime})^{\prime}=(f^{\prime
}f^{\prime}+ff^{\prime \prime})^{\prime}=2f^{\prime \prime}f+f^{\prime \prime
}f+ff^{(3)}=f^{(3)}f+3f^{\prime \prime}f^{\prime},
\]
whence we get
\begin{equation*}
(ff^{\prime})^{\prime \prime}=0.
\end{equation*}
It follows that $(ff^{\prime})^{\prime}=c$, whence we have
$ff^{\prime}=c_{1}u+d_{1}$, and so, we get $(f^{2})^{\prime}=cu+d$. 
Thus, we obtain
\[
f^{2}=cu^{2}+du+k.
\]

Taking the first and the second derivative, we have
\[
f^{\prime}=\frac{1}{2}\frac{2cu+d}{\sqrt{cu^{2}+du+k}}\  \  \  \text{and}%
\  \  \ f^{\prime \prime}=\frac{4ck-d^{2}}{4(cu^{2}+du+k)^{3/2}}.
\]
Thus
\[
\frac{f^{\prime \prime}}{f}=\frac{4ck-d^{2}}{4(cu^{2}+du+k)^{2}}%
\  \  \  \text{and}\  \  \  \frac{f^{\prime}}{f}=\frac{2cu+d}{2(cu^{2}+du+k)}.
\]
Since
\[
\frac{f^{\prime \prime}}{f}=(a^{\prime})^{2}= C^2 
e^{-4\int \frac{f^{\prime}}%
{f}du},
\]
we have
\[
\frac{f^{\prime \prime}}{f}=\frac{C^2}{f^{4}}.
\]
Thus, we obtain
\[
\frac{4ck-d^{2}}{4(cu^{2}+du+k)^{2}}=\frac{C^2}{(cu^{2}+du+k)^{2}},
\]
and therefore
\[
C^2 =\frac{4ck-d^{2}}{4}.
\]

Let $\Delta=d^{2}-4ck$ be the discriminant of $cu^{2}+du+k$. Thus, we
get 
\[
C=\frac{\sqrt{-\Delta}}{2}.
\]
Furthermore, we have
\[
a^{\prime}=Ce^{-2\int \frac{f^{\prime}}{f}du}=\frac{\sqrt{-\Delta}/2}{f^{2}%
}=\frac{\sqrt{-\Delta}/2}{cu^{2}+du+k}%
\]
and thus
\[
a=\int a^{\prime}du=\int \frac{\sqrt{-\Delta}/2}{cu^{2}+du+k}du=\int \frac
{\sqrt{-\Delta}/2}{c(u+\frac{d}{2c})^{2}+(\frac{\sqrt{-\Delta}}{2c})^{2}}du.
\]
Hence, we obtain 
\begin{equation*}
a(u)=\arctan \left(  \frac{2c}{\sqrt{-\Delta}}\left(  u+\frac{d}{2c}\right)
\right).   
\end{equation*} \end{proof}

\section{Proof of Theorem 1}
Suppose that  there exists a map $\Phi:U\rightarrow
E^{2}$ having the properties $(C1)$ and $(C2)$. 
By Lemma 1,  there are
variables $\phi$ and $\omega$ which are functions of $t,$ $u$ satisfying
\begin{equation}
\left(\frac{\partial x}{\partial u},\frac{\partial y}{\partial u}\right)=(\cos \phi
,\sin \phi), \tag{1}\label{1}%
\end{equation}
\begin{equation}
\left(\frac{\partial x}{\partial t},\frac{\partial y}{\partial t}\right)=(f(u)\cos
\omega,f(u)\sin \omega). \tag{2}\label{2}%
\end{equation}
and
\[
f^{\prime \prime}(u)=\left(  \frac{\partial \omega}{\partial u}(t,u)\right)
^{2}f(u). \tag{3}\label{3}%
\]
By $(C1)$, the meridians are mapped  to straight lines, and so, we have 
\begin{equation*}
\frac{\partial \phi}{\partial u}=0. 
\end{equation*}
Thus (\ref{1}) yields
\begin{eqnarray*}
\lefteqn{\frac{\partial}{\partial u}\left(  \frac{\partial x}{\partial u}%
,\frac{\partial y}{\partial u}\right)  =}\\
& &  \left(  \frac{\partial^{2}x}{\partial
u^{2}},\frac{\partial^{2}y}{\partial u^{2}}\right)= \frac{\partial}{\partial
u}(\cos \phi,\sin \phi)=
\left(-\sin \phi \frac{\partial \phi}{\partial u},\cos \phi
\frac{\partial \phi}{\partial u}\right)=(0,0).
\end{eqnarray*}
It follows
\begin{equation}
x(t,u)=ug_{1}(t)+g_{2}(t),\  \  \  \ y(t,u)=uh_{1}(t)+h_{2}(t). \tag{4}\label{4}%
\end{equation}
Therefore, the function $(\partial \omega/\partial u)$ is a function depending
only on the variable $u$, and hence there exist functions $a(u)$ and $b(t)$
such that
\begin{equation}
\omega(t,u)=a(u)+b(t). \tag{5}\label{5}%
\end{equation}
Combining (\ref{2}) and (\ref{5}), we deduce
\begin{equation}
\left(  \frac{\partial x}{\partial t},\frac{\partial y}{\partial t}\right)
=(f\cos(a(u)+b(t)),f\sin(a(u)+b(t))), \tag{6}\label{6}%
\end{equation}
and using that
\[
\frac{\partial^{2}x}{\partial u\partial t}=\frac{\partial^{2}x}{\partial
t\partial u},
\]
(\ref{4}) and (\ref{6}) implies that
\[
\frac{\partial}{\partial u}f\cos(a(u)+b(t))=\frac{\partial}{\partial t}%
g_{1}(t).
\]
Therefore, for each $t$ and $u$, we deduce
\begin{equation}
f^{\prime}(u)\cos((a(u)+b(t))-f(u)\sin(a(u)+b(t))a^{\prime}(u)=g_{1}^{\prime
}(t). \tag{7}\label{7}%
\end{equation}
Similarly, from
\[
\frac{\partial^{2}y}{\partial u\partial t}=\frac{\partial^{2}y}{\partial
t\partial u}%
\]
we get
\begin{equation}
f^{\prime}(u)\sin((a(u)+b(t))+f(u)\cos(a(u)+b(t))a^{\prime}(u)=h_{1}^{\prime
}(t), \tag{8}\label{8}%
\end{equation}
for each $t$ and $u$.

By taking the derivative of (\ref{7}) with respect to $u$ we have
\[
f^{\prime \prime}\cos \omega-2f^{\prime}a^{\prime}\sin \omega-f(a^{\prime}%
)^{2}\cos \omega-fa^{\prime \prime}\sin \omega=0
\]
and so, we get
\begin{equation*}
\sin \omega(2f^{\prime}a^{\prime}+fa^{\prime \prime})-(f^{\prime \prime
}-f(a^{\prime})^{2})\cos \omega=0  
\end{equation*}
Assuming that $\sin \omega \neq0$,   we obtain 
\begin{equation}
2f^{\prime}a^{\prime}+fa^{\prime \prime}=0. \tag{9}\label{9}%
\end{equation}
If $\sin \omega=0,$ then $\cos \omega \neq0$.
Thus, by taking the derivative of (\ref{8}) we can derive the same
differential equation (\ref{9}), restricting if necessary the domain where
the functions $f$ and $a$ are defined.
Lemma 2 implies that
$$f^{2}=cu^{2}+du+k$$
and
\[
a=a(u)=\arctan \left(\frac{2c}{\sqrt{-\Delta}}\left(u+\frac{d}{2c}\right)
\right), \tag{10}\label{10}
\]
where $\Delta=d^{2}-4ck$.

 Using (\ref{9}) and (\ref{10}) we get
\[
\frac{f^{\prime}}{f}=-\frac{a^{\prime \prime}}{2a^{\prime}}=a^{\prime}\tan a,
\]
which is equivalent to
 \begin{equation}
 f^{\prime}\cos a-fa^{\prime}\sin a=0.
 \tag{11}\label{11}
 \end{equation}

If $f \cos a =0$, then the above equality implies that 
$fa^{\prime}\sin a=0$. Since $f(u) a^{\prime}(u) \neq 0$, for every $u$,
we have $\sin a=0$ which is a contradiction. Thus, dividing the above
equality by  $f \cos a$, we obtain  $\frac{f^{\prime}}{f}\tan
a+a^{\prime}=0$. Substituting $f^{\prime}/f$ by $a^{\prime}\tan a$ we deduce
 $a^{\prime}(\tan a)^2+a^{\prime}=0$, whence $(\tan a)^2 = -1$
 which is a contradiction. Hence we have 
$$f^{\prime}\sin a+fa^{\prime}\cos a\neq 0.$$

On the other hand, by taking the derivative of
$f^{\prime}\sin a+fa^{\prime}\cos a$, we have
\[
(f^{\prime}\sin a+fa^{\prime}\cos a)^{\prime}=f^{\prime \prime}\sin
a+f^{\prime}a^{\prime}\cos a+f^{\prime}a^{\prime}\cos a+fa^{\prime \prime}\cos
a-f(a^{\prime})^{2}\sin a.
\]
In order to prove that this expression is zero, it suffices to show that
\[
\frac{f^{\prime \prime}}{f}\tan a+2\frac{f^{\prime}}{f}a^{\prime}+a^{\prime \prime
}-(a^{\prime})^{2}\tan a=0
\]
and one can verify, by a simple replacement, that this relation holds. Furthermore, we have
\[
a^{\prime}(0)=\frac{\sqrt{-\Delta}}{2k},\text{ }f^{\prime}(0)=\frac{d}%
{2\sqrt{k}},\text{ }f(0)=\sqrt{k},\text{ }\tan a(0)=\frac{d}{\sqrt{-\Delta}}.
\]
Then, we obtain
\begin{equation}
f^{\prime}\sin a+fa^{\prime}\cos a=\sqrt{c}. \tag{12}\label{12}
\end{equation}

 By expanding relation (\ref{7}), we obtain
\begin{eqnarray*}
\lefteqn{f^{\prime}(\cos a(u)\cos b(t)-\sin a(u)\sin b(t))-}\\
& & fa^{\prime}(\sin a(u)\cos
b(t)+\sin b(t)\cos a(u))=g_1^{\prime}(t),
\end{eqnarray*}
and from (\ref{11}), (\ref{12}) the relation $g_{1}^{\prime}(t)=\sqrt{c}\sin b(t)$ follows.

Similarly, from (\ref{9}) we have:
\[
f^{\prime}(\sin a\cos b+\sin b\cos a)+fa^{\prime}(\cos a\cos b-\sin a\sin
b)=h_{1}^{\prime}(t).
\]
whence
\[
\cos b(f^{\prime}\sin a+fa^{\prime}\cos a)+\sin b(f^{\prime}\cos a-fa^{\prime
}\sin a)=h_{1}^{\prime}(t)
\]
and so, we obtain  $h_{1}^{\prime}(t)=\sqrt{c}\, \cos b(t)$.
 Therefore, we get
$$g_{1}^{\prime}(t)=\sqrt{c}\sin b(t) \ \ \ \text{and}\ \ \  h_{1}^{\prime
}(t)=\sqrt{c}\, \cos b(t).$$

We will proceed now with the computation of the projection $\Phi.\medskip$
By hypothesis, we have, that $(\partial \phi/\partial u)=0.$ 
 Hence $\phi$ is a
function only of $t.$ From (\ref{1}), (\ref{2}) and (\ref{4}) we have that
\begin{equation}
\left(  \frac{\partial x}{\partial u},\frac{\partial y}{\partial u}\right)
=(\cos \phi(t),\sin \phi(t))=(g_{1},h_{1})\text{ }\tag{13}\label{13}%
\end{equation}
and%
\begin{equation}
\left(  \frac{\partial x}{\partial t},\frac{\partial y}{\partial t}\right)
=(f(u)\cos(a(u)+b(t)),f(u)\sin(a(u)+b(t)))=(ug_{1}^{\prime}+g_{2}^{\prime
},uh_{1}^{\prime}+h_{2}^{\prime}).\tag{14}\label{14}%
\end{equation}
Consequently, (\ref{13}) and (\ref{14}) we get respectively that
\[
(g_{1})^{2}+(h_{1})^{2}=1
\]
and
\[
u^{2}((g_{1}^{\prime})^{2}+(h_{1}^{\prime})^{2})+2u(g_{1}^{\prime}%
g_{2}^{\prime}+h_{1}^{\prime}h_{2}^{\prime})+(g_{2}^{\prime})^{2}%
+(h_{2}^{\prime})^{2}=cu^{2}+du+k.
\]
Therefore, we have:
\begin{align*}
(g_{2}^{\prime})^{2}+(h_{2}^{\prime})^{2} &  =k\\
2(g_{1}^{\prime}g_{2}^{\prime}+h_{1}^{\prime}h_{2}^{\prime}) &  =d.
\end{align*}
From the first of the previous relations we deduce that there exists a
function $r(t)$ such that
\begin{equation}
(g_{2}^{\prime},h_{2}^{\prime})=(\sqrt{k}\cos r(t),\sqrt{k}\sin r(t))\tag{15}%
\label{15}%
\end{equation}
while from the second, in combination with (\ref{13}) and (\ref{12}), we deduce
that $2\sqrt{ck}\sin(b+r)=d$ and thus
\[
\sin(b+r)=\frac{d}{2\sqrt{ck}}.
\]
Therefore, there exists real number $\theta_{0}$ such that
$$r(t)=\theta_{0}-b(t). $$
Furthermore, from (\ref{13}) we have that
\[
(g_{1}^{\prime},h_{1}^{\prime})=(-\phi^{\prime}\sin \phi,\phi^{\prime}\cos
\phi)=(\sqrt{c}\, \sin b,\text{ }\sqrt{c}\, \cos b),
\]
and so, we have the following  two cases:%

a) $\phi^{\prime}=\sqrt{c}$ and  $\phi(t)=-b(t)$. 
Thus, we have 
$$b^{\prime}(t)=-\phi^{\prime}=-\sqrt{c},$$
whence  
$$b(t)=-\sqrt{c}t+c_{0}. $$
Then, we get
\begin{equation}
(g_{1},h_{1})=(\cos(-b(t)),\sin(-b(t)))=(\cos b(t),-\sin b(t)).\tag{16}%
\label{16}%
\end{equation}
Thus, combining (\ref{4}), (\ref{15}) and (\ref{16}) we
deduce
\begin{align*}
x(t,u) &  =u\cos(b(t))+\int \sqrt{k}\cos(\theta_{0}-b(t))dt\\
y(t,u) &  =-u\sin(b(t))+\int \sqrt{k}\sin(\theta_{0}-b(t))dt.
\end{align*}

b) $\phi^{\prime}=-\sqrt{c}$ and $\phi(t)=\pi-b(t)$. Proceeding
as above, we deduce the result.
 
Furthermore, substituting $b(t)$
in the integrals above we may calculate them and thus we may find explicit
formulas for the map $\Phi.$

Conversely, by substituting the above expressions of $x(t,u)$ and $y(t,u)$ into (iii) and (iv), and supposing that $f^{2}=cu^{2}+du+k$, we see that condition $(C1)$ is easily verified. Also, condition $(C2)$  is  satisfied, since  
$\frac{\partial x}{\partial u}=\cos \phi$
implies that $$\phi=\arccos(\frac{\partial x}{\partial u})$$ and so, by taking the derivative with respect to $u$, we obtain that 
$\frac{\partial \phi}{\partial u}=0.$
Hence, Theorem \ref{main theorem} is proven.

\section{Proof of Corollary 1}
\
(1) The Gaussian curvature of each point of $S$ is given by the formula
\[
K=-\frac{f^{\prime \prime}}{f}%
\]
(see Formula (9), p. 162, in the Example 4 of \cite{[DoCarmo]}). On the other
hand, in the proof of Lemma 2
 we have shown that $f^{\prime \prime
}/f>0.$ Therefore, $K<0$ at every point of $S$ and thus statement (1) is proven.

(2) The  surfaces of revolution of constant negative curvature are well known.
A description of them can be found for example in (\cite{[Gray]}, Theorem
15.22, p. 477). Obviously these surfaces of revolution $R$ does not have the
form of the surface $S$ given in Theorem \ref{main theorem}. Therefore, for
any point $p\in R$ and for any neighborhood $U\subset R$ of $p$ there does not
exist a map $\Phi:U\rightarrow E^{2}$ satisfying the conditions $(C1)$ and
$(C2).$ On the other hand, if $S_{0}$ is a Riemannian surface of constant
negative curvature $k<0,$ it is well known that $S_{0}$ is locally isometric
to surface of revolution $R$ of constant curvature $k.$   Therefore our
statement follows.

\bigskip

Charalampos Charitos

Department of Natural Resources Management \& Agricultural Engineering

Agricultural University Athens

Iera Odos 55, 11855 Athens, Greece

email: bakis@aua.gr

\bigskip

Petroula Dospra

Department of Electrical and Computer Engineering

University of Western Makedonia

50100 Kozani, Greece

email: petroula.dospra@gmail.com

\end{document}